\documentclass[amssymb,amstex,amsmath,11pt]{article}
\usepackage[margin=1.4in]{geometry}
\usepackage{amsmath,amssymb,latexsym,amsfonts,url}
\usepackage{cite}
\usepackage{graphicx}
\usepackage{amsthm}
\usepackage{tikz}
\usepackage{pgfplots}
\usepackage{epstopdf}
\usepackage{enumerate}
\usetikzlibrary{intersections, pgfplots.fillbetween}
\usepackage[all]{xy}
\usepackage{soul}
\pgfplotsset{compat=newest}
\pgfdeclarelayer{pre main}

\numberwithin{equation}{section}

\renewcommand{\epsilon}{\varepsilon}
\theoremstyle{plain}
\newtheorem{thm}{Theorem}[section]

\newtheorem{cor}[thm]{Corollary}
\newtheorem{lem}[thm]{Lemma}
\newtheorem{remark}[thm]{Remark}
\newtheorem{p}[thm]{Problem}

\newtheorem*{theorem*}{Theorem}
\newtheorem*{proposition*}{Proposition}
\theoremstyle{definition}

\theoremstyle{remark}

\def\RR{{\mathbb R}}

\def\<{\left\langle}
\def\>{\right\rangle}

\def\RR{{\mathbb R}}

\def\({\left(}
\def\){\right)}

\begin{document}
\begin{title}
{A Willmore-type inequality for hypersurfaces with asymptotic or integral Ricci curvature bounds}
\end{title}
\begin{author}{Jihye Lee}\end{author}

\date{\today}

\maketitle

\vspace{-1cm}
\begin{abstract}
\noindent 
We establish Willmore-type inequalities for bounded domains in complete non-compact Riemannian manifolds, under either asymptotic or integral Ricci curvature bounds. Those results recover a recent inequality of Jin-Yin \cite[Theorem 1.3]{jin-yin}.
\\


\end{abstract}

%%%%%%%%%%%%%%%%%%%%%%%%%%%%%%%%%%%%%%%%%%%%%%%%%%%%%%%%%%%%%%%%%%%%%%%%%%%%%%%%%%%%
\section{Introduction}
%%%%%%%%%%%%%%%%%%%%%%%%%%%%%%%%%%%%%%%%%%%%%%%%%%%%%%%%%%%%%%%%%%%%%%%%%%

In differential geometry, the mean curvature plays a central role in the study of the extrinsic geometry of submanifolds, appearing naturally in variational problems and geometric flows. 
The Willmore functional, defined by
$\int_\Sigma \left( \frac{H}{2} \right)^2 d \sigma$,
is a classical energy functional associated with mean curvature $H$, measuring the bending of a surface.
In 1968, Willmore \cite{Willmore68} proved that any compact, closed surface in $\RR^3$ satisfies 
$$
\int_\Sigma \left(\frac{H}{2}\right)^2 d \sigma \geq 16 \pi
$$
with equality if and only if $\Sigma$ is a round sphere.
This result was extended to higher-dimensional Euclidean spaces by Chen \cite{Chen71}, and more recently to manifolds with nonnegative Ricci curvature by Agostiniani, Fogagnolo, and Mazzieri \cite{AFM2020}, using a monotonicity formula arising from potential theory. Wang \cite{XW_willmore} later gave an alternative proof of this result using tools from comparison geometry.
Motivated by these developments, various generalizations of Willmore-type inequalities have been established, including cases with asymptotically nonnegative Ricci curvature \cite{rudnik2023}, substatic manifolds \cite{BF23}, lower Ricci curvature bounds such as $\mathrm{Ric}\geq -n$ \cite{jin-yin}, and higher-codimension settings \cite{JK2025}. See also \cite{JWXZ2025} and \cite{wuwu2024} for related results.

In this paper, we generalize the Willmore inequality \cite[Theorem 1.3]{jin-yin}, which was originally established under the pointwise Ricci curvature bound $\mathrm{Ric} \geq -n$ on complete non-compact Riemannian manifolds of dimension $n+1$. For each point $x \in M$, we define the function
$$
\rho(x) := \max\{-n - \mathrm{Ric}(x), 0\},
$$
which measures the amount of Ricci curvature below $-n$ at the point $x$. 
We have 
\begin{equation}
    \mathrm{Ric}(x) \geq -n - \rho(x),
\end{equation}
which yields a Riccati inequality \eqref{ineq:riccati} for the mean curvature of tubular neighborhoods of hypersurfaces.
Using this, we derive a mean curvature comparison involving $L^p$ norm of $\rho$ (see inequality~\eqref{ineq:mean-integral-comp}), which in turn leads to a refined upper bound for the volume of tubular neighborhoods of hypersurfaces.
The main challenge lies in obtaining effective volume bounds, as existing estimates (\cite[Theorem 2]{Gallot88}, \cite[Lemma 4.1]{PetersenSprouse}) either become too large to be useful or are restricted to hypersurfaces with constant mean curvature. The associated quantity $\|\rho\|_{L^p(M)}$, often referred to as the integral Ricci curvature of order $p$, has been extensively studied in the literature. Note that when $\|\rho\|_{L^p(M)} = 0$, the pointwise condition $\mathrm{Ric} \geq -n$ is recovered. The corresponding result is presented in Theorem~1.3.

In a special case $\rho(x) \leq  \lambda(d(o,x))$, where $\lambda : [0,\infty) \to [0,\infty)$ is a non-increasing continuous function in $L^1([0,\infty))$, and $o \in M$ is a fixed reference point,
we have a pointwise mean curvature comparison involving the term $\exp\left(\int \lambda(\gamma_x(t))\,dt\right)$, where $\gamma_x(t)$ denotes the unit-speed normal geodesics emanating from $x$. This quantity remains finite under the integrability and monotonicity on $\lambda$, yielding the Willmore-type inequality presented in Theorem~1.1.

Before stating our main results, we introduce the notion of the relative volume ratio $\mathrm{RV}(\Omega)$.
Given a bounded domain $\Omega \subset M$ with smooth boundary, we define
\begin{equation}\label{def:RV}
    \mathrm{RV}(\Omega) := \liminf_{r \to \infty} \frac{\mathrm{vol}\{x \in M \,:\, d(x, \Omega) \leq r\}}{\omega_n \int_0^r (\sinh s)^n \, ds},
\end{equation}
where $\omega_n$ denotes the volume of the unit sphere in $\RR^{n+1}.$

The following theorem is the first main result of this paper.

\begin{thm}\label{thm:willmore-asymptotic}
Let $(M, g)$ be a complete non-compact Riemannian manifold of dimension $n+1$. Suppose there exists a base point $o \in M$, and let $\lambda: [0,\infty) \to [0,\infty)$  be a non-increasing continuous function in $L^1([0,\infty))$ such that
$$
\mathrm{Ric}(x) \geq -n - \lambda(d(o,x)),
$$
where $d(o,x)$ denotes the Riemannian distance from $o$ to $x$. Let $\Omega \subset M$ be a bounded domain with smooth boundary. Then
$$
\mathrm{RV}(\Omega) \cdot \omega_n \leq e^{2nb} \int_{\partial \Omega\cap\{H(x)\geq -n-2nb\}} \left(1 + 2b + \frac{1}{n}H(x)\right)^n d\mathrm{vol}_{\partial \Omega}(x),
$$
where $H(x)$ denotes the mean curvature of $\partial \Omega$ (with respect to the outward unit normal) at $x \in \partial \Omega$ and $b = \int_0^\infty \lambda(s)\, ds < \infty$.
\end{thm}

Note that the finiteness of $\mathrm{RV}(\Omega)$ follows from our assumptions in Theorem~\ref{thm:willmore-asymptotic} (see Theorem~\ref{thm:RV-asymptotic}).
We also emphasize that the assumption $\int_0^\infty \lambda(t) \,dt < \infty$ required in our setting is weaker than the condition $\int_0^\infty t\lambda(t) \,dt < \infty$, which commonly arises in the study of manifolds with asymptotically non-negative Ricci curvature (see, for instance, \cite{DLL-sobolev}, \cite{LR24}, and references therein).

\begin{remark}
When $\lambda \equiv 0$ (so that $b = 0$), the Ricci curvature lower bound reduces to $\mathrm{Ric} \geq -n$, and Theorem~\ref{thm:willmore-asymptotic} recovers the inequality
$$
\mathrm{RV}(\Omega) \cdot \omega_n \leq \int_{\partial \Omega\cap\{H(x)\geq -n\}} \left( 1+ \frac{1}{n}H(x)\right)^n \,d\mathrm{vol}_{\partial \Omega}(x),
$$
which was established in \cite[Theorem 1.3]{jin-yin}.
\end{remark}

The following theorem provides a Willmore-type inequality under an integral Ricci curvature condition.

\begin{thm}\label{thm:willmore-integral}
    Let $(M, g)$ be a complete non-compact Riemannian manifold of dimension $n+1$. Let $\Omega \subset M$ be a bounded domain with smooth mean-convex boundary and assume that
    $\mathrm{RV}(\Omega)  < \infty.$
    Let $p > \frac{n+1}{2}$. Then
    $$
        \mathrm{RV}(\Omega) \cdot \omega_n
        \leq \left(1 + \|\rho\|_p^{\frac{1}{2}} \right)
        \int_{\partial \Omega} \left(1 + \frac{H(x)}{n} \right)^n \, d\mathrm{vol}_{\partial \Omega}(x)
        + C(n,p,\|\rho\|_p) \left(1 + \frac{H(\xi)}{n} \right)^n,
    $$
    where $\displaystyle \lim_{\|\rho\|_p \to 0^+} C(n,p,\|\rho\|_p) = 0$, $H(x)$ denotes the mean curvature (with respect to the outward unit normal) at $x \in \partial \Omega$, and $\xi \in \partial \Omega$ is a point at which the mean curvature attains its maximum.
\end{thm}

\begin{remark}
    The constant $C(n,p,\|\rho\|_p)$ depends continuously on $\|\rho\|_p \in [0, \infty)$. In particular, as $\|\rho\|_{L^p(M)} \to 0^+$,  Theorem~\ref{thm:willmore-integral} recovers \cite[Theorem 1.3]{jin-yin}.
\end{remark}

We briefly comment on the assumptions in Theorem~\ref{thm:willmore-integral}.
\begin{remark}
\begin{enumerate}
    \item In the current setting, where no pointwise lower bound on Ricci curvature is imposed, we cannot guarantee a priori the finiteness of $\mathrm{RV}(\Omega)$.
    
    \item The condition $p > \frac{n+1}{2}$ in Theorem~\ref{thm:willmore-integral} reflects a common threshold that appears in the literature on integral curvature bounds. Typically, results involving integral Ricci curvature assume  $p > \frac{n}{2}$ , where  $n$  is the dimension of the ambient manifold. This lower bound on $p$  ensures that the integral control on the Ricci curvature is  enough to yield geometric or analytic consequences, such as mean curvature comparison theorem, volume comparison theorem, or heat kernel estimates (see, for example, \cite{ChenWei22,DaiWei04,petersen-wei,LF24Iso}  and the references therein). A counterexample for $p\leq\frac{n}{2}$ can be found in \cite[Proposition 9.2]{Aubry07} and \cite[Appendix]{Gallot88}. In our case, since the ambient manifold  $M$  has dimension  $n+1$, the corresponding condition becomes $p> \frac{n+1}{2}$.
    \item
    In mean curvature comparison theorems under integral Ricci curvature bounds, 
    it is common to restrict attention to regions where the mean curvature is non‑negative (cf. \cite{petersen-wei}). While analogous results can be extended to the regions with negative mean curvature (cf. \cite{Aubry07}), the associated estimates become substantially more technical.
    For simplicity, we assume that $\partial \Omega$ is mean-convex.
\end{enumerate}
\end{remark}

This paper is organized as follows: In section 2, we establish a Willmore-type inequality (Theorem~\ref{thm:willmore-asymptotic}) under asymptotic Ricci curvature conditions.
The proof proceeds by developing new comparison lemmas (Lemma~\ref{lem:comparison1} and \ref{lem:comparison2})
that provide refined estimates for solutions of linear ODEs, improving upon existing bounds in the literature.
These comparison lemmas are used to obtain pointwise mean curvature comparison along normal geodesics, which allows us to control the Jacobian determinant of the normal exponential map.
We also prove that the relative volume ratio $\mathrm{RV}(\Omega)$ is finite under our assumptions (Theorem~\ref{thm:RV-asymptotic}).
Section 3 is devoted to the proof of  Theorem~\ref{thm:willmore-integral}, which handles the more challenging case of integral Ricci curvature bounds.
We first obtain an $L^p$ version of mean curvature comparison in \eqref{ineq:mean-integral-comp}. Then, using Lemma~\ref{thm:estimate} which provides a delicate estimate for $(1+b)^p$, we control the error terms that arise from integral Ricci curvature assumption.
This ultimately leads to a Willmore-type inequality with an error term that vanishes as $\|\rho\|_p \to 0^+$, thereby recovering the classical result under pointwise Ricci curvature bounds.

\vskip 0.3cm
\noindent
{\bf Acknowledgment: }
This work is partially supported by NSF DMS grant 2403557.
The author thanks her advisor Guofang Wei for her helpful discussions and valuable guidance throughout this work.

%%%%%%%%%%%%%%%%%%%%%%%%%%%%%%%%%%%%%%%%%%%%%%%%%%%%%%%%%%%%%%%%%%%%%%%%%%%%%%%%%%%%
\section{Willmore inequality with asymptotic Ricci curvature}
%%%%%%%%%%%%%%%%%%%%%%%%%%%%%%%%%%%%%%%%%%%%%%%%%%%%%%%%%%%%%%%%%%%%%%%%%%%%%%%%%%%

In this section, we prove Theorem~\ref{thm:willmore-asymptotic} and establish several related geometric consequences. The proof strategy centers on developing refined comparison techniques for Riccati-type inequality \eqref{ineq:asymp-m} that allow us to control geometric quantities along normal geodesics. 
We begin by establishing two key comparison lemmas (Lemma~\ref{lem:comparison1} and \ref{lem:comparison2}) that provide refined estimates for solutions of second order linear ODEs, improving upon existing bounds in the literature.
These lemmas serve multiple purposes: they are essential for deriving pointwise mean curvature comparison in the proof of Theorem~\ref{thm:willmore-asymptotic} and they also enable us to prove the finiteness of the relative volume ratio $\mathrm{RV}(\Omega)$ in Theorem~\ref{thm:RV-asymptotic}.

A crucial aspect of our analysis involves investigating the relationship between mean curvature and focal radius $\tau(x)$.
This investigation, carried out in a claim within the proof of Theorem~\ref{thm:willmore-asymptotic}, leads to several important geometric consequences. In particular, we derive Corollary~\ref{coro:H-intersection}, which ensures that the integration region $\partial \Omega\cap\{H(x)\geq -n-2nb\}$ in Theorem~\ref{thm:willmore-asymptotic} is non-empty, and Corollary~\ref{coro:H-intersection2}, which provides a compactness criterion based on a lower bound for the mean curvature of the manifold's boundary.

We now prove the following lemmas whose proofs follow the ideas of \cite[Lemma 2.13]{Pigola-vanishing} and \cite[Lemmas 2.5 and 2.6]{DLL-sobolev}. Although similar comparison results appear in the literature (see, for instance, \cite[Lemma 2.6]{DLL-sobolev}), they are not directly applicable in our setting. In particular, applying \cite[Lemma 2.6]{DLL-sobolev} with $G(t)= 1 + \lambda(d(o,\gamma(t)))$ would require $G \in L^1([0,\infty))$, which does not hold under our assumptions. As a result, the corresponding upper bounds become too large for our purposes.
These lemmas provide a more suitable estimate under our assumptions.

\begin{lem}\label{lem:comparison1}
Let $\Lambda : [0, \infty) \to [0, \infty)$ be a function in $L^1([0,\infty))$. Let $\psi_1$ be the solution to the initial value problem
$$
\begin{cases}
\psi_1''(t) = (1 + \Lambda(t)) \psi_1(t), \\
\psi_1(0) = 0, \quad \psi_1'(0) = 1.
\end{cases}
$$
Then for all $t \geq 0$, we have
\begin{enumerate}[(i)]
    \item\label{item:lem-comparison1-1} $\sinh t \leq \psi_1(t) \leq \displaystyle\int_0^t e^{\int_0^s \Lambda(\tau)\, d\tau} \cosh s\, ds$.
    \item\label{item:lem-comparison1-2} $\cosh t \leq \psi_1'(t) \leq e^{\int_0^t \Lambda(\tau)\, d\tau} \cosh t$.
    \item\label{item:lem-comparison1-3} The limit $\lim\limits_{t \to \infty}\frac{\psi_1(t)}{\sinh(t)}$ exists.
\end{enumerate}
\end{lem}

\begin{proof}
Let $\phi_1(t) = \sinh t$. Then
$
\phi_1''(t) = \phi_1(t)
$
with $\phi_1(0) = 0$ and $\phi_1'(0) = 1$. By applying the second-order linear ODE comparison theorem \cite[Lemma 2.1]{Pigola-vanishing}, we obtain
$$
\coth t \leq \frac{\psi_1'(t)}{\psi_1(t)} \quad \text{and} \quad 0 \leq \sinh t \leq \psi_1(t)
$$
for all $t > 0$. Combining the inequalities above, we deduce that $\psi_1'(t) \geq \cosh t$.

To find an upper bound for $\psi_1$ and $\psi_1'$, define the function
$$
\phi_2(t) := \int_0^t e^{\int_0^s \Lambda(\tau)\, d\tau} \cosh s\, ds
$$
to serve as a supersolution.
A direct computation shows that
$$
\phi_2''(t) = (\Lambda(t) \coth t + 1) \sinh t \cdot e^{\int_0^t \Lambda(\tau)\, d\tau} \geq (\Lambda(t) + 1) \phi_2(t)
$$
with $\phi_2(0) = 0$ and $\phi_2'(0) = 1$.
Applying the comparison theorem \cite[Lemma 2.1]{Pigola-vanishing} again, we conclude that $\psi_1(t) \leq \phi_2(t)$ for all $t \geq 0$. Moreover, the inequality $\psi_1'' (t) \leq \phi_2''(t)$ with $\psi_1'(0) = 1 = \phi_2'(0)$ yields 
$$\psi_1'(t) \leq \phi_2'(t) =  e^{\int_0^t \Lambda (\tau) \, d \tau} \cosh t.$$

To prove the existence of the limit  $\lim\limits_{t \to \infty}\frac{\psi_1(t)}{\sinh(t)}$, observe that
$$(\psi_1'(t) \sinh t - \psi_1(t) \cosh t)' =\Lambda(t) \psi_1(t) \sinh t\geq 0.$$
By our initial condition we have $\psi_1'(t) \sinh t - \psi_1(t) \cosh t \geq 0$ for all $t \geq 0$.
Thus,
$$\left(\frac{\psi_1(t)}{\sinh t}\right)' = \frac{\psi_1'(t) \sinh t - \psi_1(t) \cosh t}{\sinh^2 t} \geq 0.$$
Moreover, Lemma~\ref{lem:comparison1}~\eqref{item:lem-comparison1-1} implies that for $t >0$ we have
\begin{equation}\label{ineq:lem1-exp}
    1 \leq \frac{\psi_1(t)}{\sinh t} \leq \exp\left({\int_0^\infty \Lambda (t) \,dt}\right)<\infty.
\end{equation}
Consequently, the ratio $\frac{\psi_1(t)}{\sinh(t)}$ is monotone increasing and uniformly bounded above, hence the limit $\lim\limits_{t \to \infty}\frac{\psi_1(t)}{\sinh(t)}$ exists.
\end{proof}

\begin{lem}\label{lem:comparison2}
Let $\Lambda : [0, \infty) \to [0, \infty)$ be a function in $L^1([0,\infty))$. Let $\psi_1$ and $\psi_2$ be the solutions to the initial value problems
$$
\begin{cases}
\psi_1''(t) = (1 + \Lambda(t)) \psi_1(t), \\
\psi_1(0) = 0, \quad \psi_1'(0) = 1,
\end{cases}
\qquad
\begin{cases}
\psi_2''(t) = (1 + \Lambda(t)) \psi_2(t), \\
\psi_2(0) = 1, \quad \psi_2'(0) = 0.
\end{cases}
$$
Then, 
\begin{enumerate}[(i)]
    \item $\displaystyle
\frac{\psi_2(t)}{\psi_1(t)} \leq  \coth t + \int_0^t \frac{\Lambda(s)}{\cosh^2 s}\, ds 
$ for all $t > 0$.
    \item\label{item:comp2-item2} The limit $\lim\limits_{t \to \infty}\frac{\psi_2(t)}{\psi_1(t)}$ exists and
    $$\lim\limits_{t \to \infty}\frac{\psi_2(t)}{\psi_1(t)} \leq 1 +\int_0^\infty \frac{\Lambda(s)}{\cosh^2 s}\, ds < \infty. $$
\end{enumerate}
\end{lem}

\begin{proof}
From the differential equations, we observe that
$$
(\psi_2 \psi_1' - \psi_1 \psi_2')' = \psi_2 \psi_1'' - \psi_1 \psi_2'' \equiv 0.
$$
Hence, by the initial conditions,
\begin{equation}\label{eq:lem-psi}
\psi_2(t)\psi_1'(t) - \psi_1(t)\psi_2'(t) \equiv 1.
\end{equation}
From equation \eqref{eq:lem-psi}, we compute
\begin{equation}\label{eq:lem-comp2-temp1}
    \frac{\psi_2}{\psi_1} = \frac{1 + \psi_1 \psi_2'}{\psi_1 \psi_1'} = \frac{1}{\psi_1 \psi_1'} + \frac{\psi_2'}{\psi_1'}.
\end{equation}
To estimate the second term, differentiate $\psi_2' / \psi_1'$:
\begin{equation}\label{eq:lem-comp2-temp2}
    \left( \frac{\psi_2'}{\psi_1'} \right)' = \frac{(1 + \Lambda(t))(\psi_2 \psi_1' - \psi_2' \psi_1)}{(\psi_1')^2} = \frac{1 + \Lambda(t)}{(\psi_1')^2},
\end{equation}
where we used equation \eqref{eq:lem-psi} in the last step.
Combining the inequality Lemma~\ref{lem:comparison1}~\eqref{item:lem-comparison1-2} and the identity \eqref{eq:lem-comp2-temp2}, we obtain
$$
\left( \frac{\psi_2'}{\psi_1'} \right)' \leq \frac{1 + \Lambda(t)}{\cosh^2 t} = \mathrm{sech}^2 t + \frac{\Lambda(t)}{\cosh^2 t}.
$$
Integrating both sides over $[0, t]$ yields
\begin{equation}\label{ineq:lem2-est1}
    \frac{\psi_2'(t)}{\psi_1'(t)} \leq \tanh t + \int_0^t \frac{\Lambda(s)}{\cosh^2 s}\, ds.
\end{equation}
For the remaining term, again by Lemma~\ref{lem:comparison1}, we have
\begin{equation}\label{ineq:lem2-est2}
\frac{1}{\psi_1(t) \psi_1'(t)} \leq \frac{1}{\sinh t \cosh t}.    
\end{equation}
Combining the above estimates \eqref{eq:lem-comp2-temp1}, \eqref{ineq:lem2-est1}, and \eqref{ineq:lem2-est2}, we conclude that
\begin{equation}\label{ineq:lem2-psi2psi1}
\frac{\psi_2(t)}{\psi_1(t)} \leq \tanh t + \int_0^t \frac{\Lambda(s)}{\cosh^2 s}\, ds + \frac{1}{\sinh t \cosh t} = \coth t  + \int_0^t \frac{\Lambda(s)}{\cosh^2 s}\, ds.    
\end{equation}
We now want to consider the limit of $\frac{\psi_2}{\psi_1}(t)$.
Note that
\begin{equation}\label{eq:lem-comp2-temp3}
    \left(\frac{1}{\psi_1\psi_1'}\right)'  = \frac{- (\psi_1')^2- (1+\Lambda(t))\psi_1^2}{(\psi_1\psi_1')^2}.
\end{equation}
By the identities \eqref{eq:lem-comp2-temp1}, $\eqref{eq:lem-comp2-temp2}$, and \eqref{eq:lem-comp2-temp3}, we have
$$\left(\frac{\psi_2}{\psi_1}\right)' = \frac{- (\psi_1')^2- (1+\Lambda(t))\psi_1^2}{(\psi_1\psi_1')^2} + \frac{1+\Lambda(t)}{(\psi_1'(t))^2} = - \frac{1}{(\psi_1(t))^2} <0.$$
Since $\frac{\psi_2}{\psi_1}(t)$ is nonnegative and monotone decreasing for all $t > 0$, the limit $\lim\limits_{t \to \infty}\frac{\psi_2}{\psi_1}(t)$ exists.
Moreover, the inequality \eqref{ineq:lem2-psi2psi1} implies that
$$\lim_{t \to \infty }\frac{\psi_2}{\psi_1} \leq 1 +  \int_0^\infty \frac{\Lambda(s)}{\cosh^2 s}\, ds < \infty.$$
\end{proof}

The following comparison lemma is a slight modification of \cite[Lemma 2.5]{DLL-sobolev}, formulated so that the endpoint $a$ may be finite or infinite. For completeness, we state it here in the form that will be used in the proof of Theorem~\ref{thm:willmore-asymptotic}.

\begin{lem}[{\cite[Lemma 2.5]{DLL-sobolev}}]\label{lem:G}
    Let $G$ be a continuous function on $[0,a)$ and let $\phi,\psi \in C^2([0,a))$ be solutions of the following problems
    $$\begin{cases}
        \phi'' \leq G \phi, \quad t \in (0, a),\\
        \phi(0) =1, \phi'(0) = b,
    \end{cases} \quad \begin{cases}
        \psi'' \geq G \psi, \quad t \in (0, a),\\
        \psi(0) =1, \psi'(0) = b,
    \end{cases}$$
    where $a,b$ are constants and $0 < a\leq +\infty$.
    If $\phi (t) >0$ for $t \in (0,T) \subseteq (0, a)$, then $\psi(t) >0$ in $(0, T)$ and
    $$\frac{\phi'}{\phi} \leq \frac{\psi'}{\psi} \quad \mbox{ and } \quad \phi \leq \psi \quad \mbox{ on } (0,T).$$
\end{lem}

We are now ready to prove Theorem~\ref{thm:willmore-asymptotic}.
\begin{proof}[Proof of Theorem \ref{thm:willmore-asymptotic}]
    Let $\Sigma := \partial \Omega$ and $\nu$ be the outward unit normal vector field on $\Sigma$.
    For each $x \in \Sigma$, consider the normal geodesic $\gamma_x(t) := \exp_x (t  \nu (x))$. Define the focal radius of $\Sigma$ at $x \in \Sigma$ by
    \begin{equation}\label{def:focalRadius}
        \tau(x) := \sup\{L>0\,|\, \gamma_x(t) \text{ is minimizing on } [0,L]\}.
    \end{equation}
    Then $0< \tau (x) \leq \infty$ and the function $\tau(x)$ is continuous on $\Sigma$.
    The cut locus of $\Sigma$ is given by
    $$\mathrm{Cut}(\Sigma)=\{\exp_x(\tau(x) \nu(x))\,|\, x \in \Sigma\}.$$
    Then the map
    $$\Phi : \{(x,r) \, | \, x \in \Sigma, 0\leq r <\tau (x)\} \to ( M \setminus \Omega) \setminus \mathrm{Cut}(\Sigma), \quad \Phi(x,r) = \exp_x (r \nu(x))$$
    is a diffeomorphism. 
    
    To apply the area formula, we estimate $\mathrm{det}(D\Phi)(x,r)$. Fix a point $(x,r)$ in the domain of $\Phi$.
    Let $\{e_1,\ldots, e_n\}$ be an orthonormal basis of $T_x \Sigma$.
    For each $1 \leq i \leq n$,
    define a vector field $X_i(t)$ along the normal geodesic $\gamma_x(t)$ by
    \begin{align*}
        X_i(t)  := D\Phi(e_i)(x,t)= \left.\frac{\partial}{\partial s}\right|_{s=0}\exp_{\exp_x^{\Sigma} (s e_i)}(t  \nu(\exp_x^{\Sigma} (s  e_i)))
    \end{align*}
    for $0 \leq t < \tau(x)$.
     Then $X_i(t)$ is a Jacobi field along  $\gamma_x(t)$ with initial conditions
    \begin{align*}
        X_i(0) = e_i, \quad 
        X_i'(0)= \nabla_{e_i}\nu(x).
    \end{align*}
    Moreover, $X_i(t)$ remains orthogonal to $\gamma'(t)$ for all $t \in (0, \tau(x))$, as it arises from a normal variation of geodesics.
    Define an another Jacobi field $X_{n+1}(t)$ along $\gamma_x(t)$ by $X_{n+1}(t) = \gamma_x'(t)$ for $0 \leq t < \tau(x)$.
     
    For each $ 1 \leq i \leq n$, let $E_i(t)$ denote the parallel transport of $e_i$ along $\gamma_x(t)$ for $0 \leq t \leq r$ and define $E_{n+1}(t) := \gamma_x'(t)$. Then $\{E_i (t)\}_{i=1}^{n+1}$ forms a basis of $T_{\gamma_x(t)}M$. Set $X_{ij}(t) = \langle X_i(t), E_j(t)\rangle$. 
    For $1 \leq i \leq n$, we have
    $$X_{n+1, i}(t) = 0, \quad X_{i,n+1}(t) =0, \quad X_{n+1,n+1}(t) = 1.$$
    This implies that $\det D\Phi$ reduces to the determinant of its $n \times n$ tangent block due to the block-diagonal structure:
    $$
        \det D\Phi = \det(X_{ij})_{(n+1)\times (n+1)} = \det(X_{ij})_{n \times n}.
    $$
    
    Since $X_i(t)$ is a Jacobi field, it satisfies
    $$X_i''(t) + R(X_i(t), \gamma_x'(t))\gamma_x'(t)=0.$$
    Then we obtain
    \begin{align*}
        0 = X_{ij}''(t) + \langle R(E_k(t), \gamma_x'(t))\gamma_x'(t), E_j(t)\rangle X_{ik}(t).
    \end{align*}
    Let $X(t)$ and $S(t)$ be $n\times n$ matrices whose entries are given by $X_{ij}(t)$ and $\langle R(E_k(t), \gamma_x'(t))\gamma_x'(t), E_j(t)\rangle$, respectively. Then, in matrix notation, the equation becomes $$X''(t) + X(t) S(t)= 0.$$
    Define $Y(t) := X^{-1}(t) X'(t).$ Then
    $$Y' (t) =  -S(t) - (Y(t))^2,$$
    with an initial condition $Y(0) = X'(0).$
    Define $m(t) := \mathrm{tr} (Y(t))$. Then
    \begin{equation}\label{ineq:asymp-m}
        m'(t) = - \mathrm{tr}S(t) - \mathrm{tr}(Y(t)^2) \leq  -\mathrm{Ric}(\gamma_x'(t), \gamma_x'(t)) - \frac{1}{n}(m(t))^2,
    \end{equation}
    where the inequality follows from $\mathrm{tr}(A^2) \geq \frac{1}{n}(\mathrm{tr}A)^2$ for a $n\times n$ symmetric matrix.
    The initial condition is
    $m(0)= H(x),$ where $H(x)$ denotes the mean curvature of $\Sigma$ at $x$. Geometrically, $m(t)$ represents the mean curvature of the hypersurface $\Sigma_{t} := \{\Phi(x,t) \, | \, x \in \Sigma\}$.

    From the Ricci curvature lower bound, it follows that
    $$m'(t) + \frac{1}{n}(m(t))^2 \leq - \mathrm{Ric}(\gamma_x'(t), \gamma_x'(t)) \leq n+n\lambda(d(o,\gamma_x(t))).$$
    Set $g(t) := \frac{1}{n}m(t)$. Then
    the inequality becomes
    $$g'(t) + g^2(t)  \leq 1 + \lambda(d(o,\gamma_x(t))), \quad g(0) = \frac{1}{n}H(x).$$
    Define $$\phi(t) =\exp\left({\int_0^tg(s) \,ds}\right), \quad 0 \leq t < \tau(x).$$ Then
    $$\phi''(t)  \leq \phi (t) \left( 1 + \lambda(d(o,\gamma_x(t)))\right).$$
    The initial conditions are
    $\phi(0) = 1$, $\phi'(0) =  \frac{1}{n}H(x)$.
    Let $\psi_1(t)$ and $\psi_2(t)$ denote the solutions to the initial value problems:
    \begin{equation}\label{df:psi1psi2}
        \begin{cases}
        \psi_1'' (t) =   \left( 1 +\lambda(d(o,\gamma_x(t)))\right) \psi_1(t),\\
        \psi_1(0) = 0, \quad \psi_1'(0) =1,
    \end{cases} \quad \begin{cases}
        \psi_2'' = \left( 1 +\lambda(d(o,\gamma_x(t)))\right) \psi_2(t),\\
        \psi_2 ( 0 ) = 1, \quad \psi_2'(0) = 0.
    \end{cases}
    \end{equation}
    The existence and uniqueness of $\psi_1$ and $\psi_2$ are guaranteed by linear ODE theory.
    A function $\Lambda(t) = \lambda(d(o,\gamma_x(t)))$ is $L^1([0,\infty))$ because
    $$\int_0^\infty \Lambda(t) \, dt \leq \int_0^\infty \lambda(|d(o,x) - t|)\,dt =\lim_{s \to \infty} \left( \int_0^{d(o,x)} \lambda(u)\,du + \int_0^{s-d(o,x)} \lambda(u) \, du \right) \leq 2b,$$
    where we used the triangle inequality and monotonicity of the function $\lambda(t)$ in the first inequality and $b = \int_0^\infty \lambda(s) \, ds<\infty$.
    By Lemmas \ref{lem:comparison1} and  \ref{lem:comparison2} with $\Lambda(t) = \lambda(d(o,\gamma_x(t)))$, it follows that
    \begin{equation}\label{ineq:psi-fraction}
        \frac{\psi_2}{\psi_1}(t) \leq \coth t + \int_0^t \frac{\lambda(d(o,\gamma_x(s)))}{\cosh^2 s}\, ds  \leq \coth t + 2b,
    \end{equation}
    and
    \begin{align}\label{ineq:psi1}
        \psi_1 (t)  \leq \int_0^t \exp\left({\int_0^s  \lambda(d(o,\gamma_x(\tau))) \, d \tau } \right) \cosh s\, ds \leq e^{2b}\sinh t.
    \end{align}
    Define \begin{equation}\label{df:psi}
        \psi(t) := \psi_2(t) + \frac{1}{n}H(x) \psi_1(t).
    \end{equation} Then 
    $$\psi''(t) = (1+\lambda(d(o,\gamma_x(t))))\psi(t), \quad 
    \psi(0) = 1 ,  \quad  \psi'(0) = \frac{1}{n}H(x).$$
    By Lemma~\ref{lem:G}, it follows that $\psi(t) \geq 0$ and 
    \begin{equation}\label{ineq:g(t)bound}
        \frac{1}{n}m(t) =g(t) = \frac{\phi'}{\phi} \leq \frac{\psi'}{\psi} \quad \mbox{ on } (0, \tau(x)).
    \end{equation}
    Thus,
    \begin{equation}\label{eq:detxt-monotone}
        \frac{d}{dt} \log(\mathrm{det} X(t)) = \operatorname{tr}(Y(t)) = m(t) \leq n \cdot\frac{\psi' }{\psi} = \frac{d}{dt}\log (\psi(t)^n).
    \end{equation}
    By integrating the inequality \eqref{eq:detxt-monotone} and combining the inequalities \eqref{ineq:psi-fraction} and \eqref{ineq:psi1}, we obtain
    \begin{align}
        \mathrm{det}X(t) &\leq \psi (t)^n = \left(\psi_2(t) + \frac{1}{n}H(x) \psi_1(t)\right)^n = \left( \frac{\psi_2(t)}{\psi_1(t)} + \frac{1}{n}H(x)\right)^n \psi_1(t)^n\nonumber\\
        & \leq \left( \cosh t + \left(2b + \frac{1}{n}H(x)\right)\sinh t\right)^ne^{2nb} \label{ineq:det-estimation}
    \end{align}
    for all $0< t < \tau (x)$.

    Define 
    \begin{equation}\label{def:jxt}
        J(x,t) := \begin{cases}
        |\mathrm{det} X(t)|, & \mbox{if }t < \tau(x),\\
        0, & \mbox{if }t \geq \tau (x).
    \end{cases}
    \end{equation}
    \paragraph{Claim.} If $H(x) < - n - 2nb$, then $\tau(x) <\infty$.
    \begin{proof}[Proof of claim]
        Assume, for contradiction, that $\tau(x) = \infty$. Then the mean curvature $m(t)$ of the parallel hypersurfaces $\Sigma_t$ is well-defined for all $t>0$.
        Since
        $$0 < - \frac{1}{2b + \frac{1}{n}H(x)} < 1,$$
        there exists $t_0>0$ such that $$\tanh (t_0) > - \frac{1}{2b + \frac{1}{n}H(x)}.$$
        Then
        $$\coth(t_0) + 2b+ \frac{1}{n}H(x) <0.$$
        Recall that 
        \begin{align*}
            \psi(t_0) & = \left(\frac{\psi_2(t_0)}{\psi_1(t_0)} + \frac{1}{n}H(x)\right) \psi_1 (t_0).
        \end{align*}
        Since $\psi_1 (t_0) >0$, the inequality \eqref{ineq:psi-fraction} implies
        $$\psi (t_0) \leq \left(\coth t_0 + 2b + \frac{1}{n} H(x)\right) \psi_1(t_0)<0.$$
        On the other hand, since $\psi(0) = 1$, it follows that there exists $t_1\in (0, t_0)$ such that $\psi(t_1) = 0$, $\psi(t) >0$ for all $0\leq t<t_1$, and $\psi'(t_1) <0$. Therefore,
        $$\lim_{t \to t_1^-} \frac{\psi'(t)}{\psi(t)} = -\infty.$$
        From the inequality \eqref{ineq:g(t)bound}, we conclude that 
        $$\lim_{t \to t_1^-} m(t) \leq \lim_{t \to t_1^-} n \cdot \frac{\psi'(t)}{\psi(t)} =-\infty.$$
        This shows that $m(t) \to -\infty$ as $t \to t_1^-$, contradicting the assumption that the mean curvature $m(t)$ of $\Sigma_t$ remains well-defined for all $t \geq 0$.
    \end{proof}

    From the claim, we observe that
    \begin{align*}
        0 &\leq \lim_{r \to \infty} \frac{1}{r} \int_{\Sigma \cap \{H(x) < - n -2nb\}}\int_0^r J(x,t)\, dt \, d \mathrm{vol}_\Sigma (x)\\
        & = \lim_{r \to \infty} \frac{1}{r} \int_{\Sigma \cap \{H(x) < - n -2nb\}}\int_0^{\tau(x)} |\mathrm{det}X(t)|\, dt \, d \mathrm{vol}_\Sigma (x)=0.
    \end{align*}
    Thus,
    $$\int_{\Sigma \cap \{H(x) < - n -2nb\}}\int_0^r J(x,t)\, dt \, d \mathrm{vol}_\Sigma (x) = o(r) \quad \text{ as } r \to \infty.$$

    Since the map $$\Phi : \{(x,r) \, | \, x \in \Sigma, 0\leq r <\tau (x)\} \to ( M \setminus \Omega) \setminus \mathrm{Cut}(\Sigma)$$ is a diffeomorphism, we have
    {\allowdisplaybreaks
    \begin{align*}
        &\mathrm{vol}(\{x \in M \setminus \Omega \, | \, d(x, \Omega) \leq r\})\\
        & = \int_\Sigma \int_{0}^{\min\{r,\tau (x)\}} |\mathrm{det}D\Phi (x,t)| \,\mathrm{d}t \, d\mathrm{vol}_\Sigma (x) \\
        & \leq \int_\Sigma \int_{0}^{r} J(x,t) \,\mathrm{d}t \, d\mathrm{vol}_\Sigma (x)\\
        & = \int_{\Sigma\cap\{H(x)\geq -n-2nb\}} \int_{0}^{r} J(x,t) \,\mathrm{d}t \, d\mathrm{vol}_\Sigma (x) + \int_{\Sigma\cap\{H(x)< -n-2nb\}} \int_{0}^{r} J(x,t) \,\mathrm{d}t \, d\mathrm{vol}_\Sigma (x) \\
        & \leq \int_{\Sigma\cap\{H(x)\geq -n-2nb\}} \int_{0}^{r}  \left[\left( \cosh t+\left( 2b + \frac{1}{n}H(x)\right)\sinh t\right)^n e^{2nb} \right] \,\mathrm{d}t \, d\mathrm{vol}_\Sigma (x) + o(r)\\
        & =\int_{\Sigma\cap\{H(x)\geq -n-2nb\}} \left[e^{2nb}\left(1 + 2b  + \frac{1}{n}H(x)\right)^n\frac{e^{nr}}{n \cdot 2^n} + O\left(e^{(n-2)r}\right)\right]  \, d\mathrm{vol}_\Sigma (x)+ o(r)\\
        & = e^{2nb} \cdot \frac{e^{nr}}{n \cdot 2^n} \int_{\Sigma\cap\{H(x)\geq -n-2nb\}} \left(1 + 2b  + \frac{1}{n}H(x)\right)^n d\mathrm{vol}_\Sigma (x) + O\left(e^{(n-2)r}\right) + o(r).
    \end{align*}}
    Dividing both sides by $\frac{e^{nr}}{n \cdot 2^n}$ and taking the limit as $r \to \infty$, we obtain
    $$\mathrm{RV}(\Omega) \cdot \omega_n \leq e^{2nb}\int_{\Sigma\cap\{H(x)\geq -n-2nb\}} \left(1 + 2b  + \frac{1}{n}H(x)\right)^n d\mathrm{vol}_\Sigma (x).$$
\end{proof}

By the claim in the above proof, we have the following corollary.

\begin{cor}\label{coro:H-intersection}
    Let $M^{n+1}$ be a complete non-compact Riemaninan manifold with 
    $$\mathrm{Ric}(x) \geq -n -n\lambda(d(o,x)),$$
    where $\lambda:[0,\infty) \to [0,\infty)$ is a monotone decreasing continuous function and $b = \int_0^\infty \lambda (t) \, dt < \infty$.
    Let $\Omega$ be a bounded domain with smooth boundary.
    Then
    $$\{ x \in \partial \Omega \mid H(x) \geq -n - 2nb \}\neq \emptyset.$$
\end{cor}

\begin{proof}
    Suppose, for contradiction, that $H(x) < -n - 2nb$ for every $x \in \partial \Omega$. Then, by the claim proved earlier, $\tau(x) < \infty$ for all $x \in \partial \Omega$. Since the map
    $$\Phi:\{(x,r) \,| \, x \in \partial M , 0 \leq r< \tau (x)\} \to (M\setminus\Omega) \setminus \mathrm{Cut}(\partial M),$$
    defined by $\Phi(x,r) = \exp_x (r \nu(x))$, is a diffeomorphism onto its image, $M \setminus \Omega$ is bounded. This contradicts the assumption that $M$ is non-compact.
\end{proof}

The following result is a consequence of Corollary~\ref{coro:H-intersection}, and generalizes \cite[Proposition 2]{Galloway} and  \cite[item (3) of Theorem C]{kaxue}.

\begin{cor}\label{coro:H-intersection2}
    Let $M^{n+1}$ be a complete Riemannian manifold with smooth boundary $\partial M$.
  Suppose there exists a monotone decreasing continuous function $\lambda:[0,\infty) \to [0,\infty)$ such that $b = \int_0^\infty \lambda(s) \, ds <\infty$ and
    $$\mathrm{Ric}(x) \geq -n - n\lambda(d(o,x)),$$
    for some base point $ o \in M$.
    If $\partial M$ is compact and $H(x) > n + 2nb$ for all $x \in \partial M$, then $M$ is compact.
\end{cor}
\begin{proof}
Suppose, for contradiction, that $M$ is not compact. We consider the inward normal exponential map
$$
\Phi : \left\{ (x, r) \in \partial M \times [0, \infty) \mid 0 \leq r < \tau(x) \right\} \to M \setminus \mathrm{Cut}(\partial M),
$$
defined by $\Phi(x, r) = \exp_x(-r \nu(x))$, where $\nu(x)$ is the outward unit normal vector at $x$, and $\tau(x)$ is the focal locus of $\gamma_x(t)= \exp_x(-t \nu(x))$.
As in the proof of Theorem~\ref{thm:willmore-asymptotic}, define $m(t): =\mathrm{tr}(Y(t))$, which represents the mean curvature of the parallel hypersurface at distance $t$ from $\partial M$. Then
$$
m'(t) + \frac{1}{n} m(t)^2 \leq n + n \lambda(d(o,\gamma_x(t))),
$$
with initial value $m(0) = -H(x) < -n - 2nb$ for all $x \in \partial M$.
By the same argument used in the claim of Theorem~\ref{thm:willmore-asymptotic}, it follows that $\tau(x) < \infty$ for all $x \in \partial M$.

Since $\partial M$ is compact and $\tau(x)$ is continuous on $\partial M$, this implies there exists $R > 0$ such that every inward normal geodesic encounters a focal point within distance $R$. Thus, the image of the inward normal exponential map is contained in a bounded subset of $M$, contradicting the assumption that $M$ is non-compact.
\end{proof}

In the remainder of this section, we prove that the limit $\mathrm{RV}(\Omega)$ is finite.

\begin{thm}\label{thm:RV-asymptotic}
    Let $(M,g)$ be a complete non-compact Riemannian manifold of dimension $n+1$.
    Let $o \in M$ be a base point, and let $\lambda : [0,\infty) \to [0,\infty)$ be a non-increasing function in $L^1([0,\infty))$ such that
    $$
        \mathrm{Ric}(x) \ge -n - n \lambda(d(o,x)),
    $$
    where $d(o,x)$ denotes the Riemannian distance from $o$ to $x$.
    Let $\Omega \subset M$ be a bounded domain with smooth boundary.
    The limit 
    $$\lim_{r \to \infty} \frac{\mathrm{vol}\{x \in M \,:\, d(x, \Omega) \leq r\}}{\omega_n \int_0^r (\sinh s)^n \, ds}$$ exists and is finite.
\end{thm}
\begin{proof}
    If $\mathrm{vol}(M) =0$, then the limit clearly vanishes. 
    We may thus assume that $\mathrm{vol}(M)$ is infinite.
    We first show that the limit
    $$
      \lim_{r \to \infty} \int_{\partial \Omega} \frac{\bar J (x,r)}{\omega_n (\sinh r)^n}d\mathrm{vol}_{\partial \Omega}
    $$
    exists and is finite, using the Dominated Convergence Theorem. 
    Then, by L'H\^{o}pital's rule, we conclude that the limit
    $$
    \lim_{r \to \infty} \frac{\operatorname{vol}\{x \in M : d(x, \Omega) \leq r\}}{\omega_n \int_0^r (\sinh s)^n \, ds}
    $$
    exists and equals the limit above.

    To analyze the integrand, observe that
    $$\frac{\bar J (x,r)}{\omega_n (\sinh r)^n} =  \frac{1}{\omega_n}\frac{\bar J (x,r)}{\psi(r)^n}\cdot \frac{\psi(r)^n}{\psi_1(r)^n} \cdot \frac{\psi_1(r)^n}{(\sinh r)^n},$$
    where $\psi_1(r)$ and $\psi(r)$ are defined in \eqref{df:psi1psi2} and \eqref{df:psi}.
    
    From inequality \eqref{eq:detxt-monotone}, the ratio $\frac{\bar J(x,r)}{\psi(r)^n}$ is monotone decreasing in $r$ and bounded between $0$ and $1$.
    Therefore, the limit $\lim\limits_{r \to \infty} \frac{\bar{J}(x, r)}{\psi(r)^n}$ exists.
    Recall that
    $\psi(r) = \psi_2(r) + \frac{1}{n}H(x) \psi_1(r)$. Then
    $$\left(\frac{\psi(r)}{\psi_1(r)}\right)^n = \left(\frac{\psi_2(r)}{\psi_1(r)} + \frac{1}{n}H(x)\right)^n \leq \left( \coth r+ 2b + \frac{1}{n}H(x)\right)^n,$$
    where $b = \int_0^\infty \lambda(t) \,dt$.
    Since Lemma~\ref{lem:comparison2}~\eqref{item:comp2-item2} shows that the limit $\lim\limits_{r \to \infty}\frac{\psi_2(r)}{\psi_1(r)}$ exists, it follows that the limit of $\left(\frac{\psi(r)}{\psi_1(r)}\right)^n$ also exists as $r \to \infty$.
    In addition, Lemma~\ref{lem:comparison1}~\eqref{item:lem-comparison1-3} and the inequality \eqref{ineq:lem1-exp} imply that
    the limit $\lim\limits_{r \to \infty} \left(\frac{\psi_1(r)}{\sinh r}\right)^n$ exists and satisfies
    $$1\leq \left(\frac{\psi_1(r)}{\sinh r}\right)^n \leq e^{2nb}.$$

    Thus, the limit $\lim\limits_{r \to \infty} \frac{\bar J (x,r)}{\omega_n (\sinh r)^n}$ exists and
     $$\frac{\bar J (x,r)}{\omega_n (\sinh r)^n} \leq  \frac{1}{\omega_n}\cdot 1\cdot \left( \coth t_0+ 2b +\frac{1}{n}H(x)\right)^n \cdot e^{2nb}$$
     for $r > t_0$ (some fixed number).
     Since $\partial \Omega$ is compact and smooth, and $H(x)$ is continuous on $\partial \Omega$, the integrand $\frac{\bar J(x,r)}{\omega_n (\sinh r)^n}$ is bounded above by $L^1(\partial \Omega)$ function.
    Hence, the dominated convergence theorem is applied and the limit 
    $$\lim_{r \to \infty} \int_{\partial \Omega} \frac{\bar J (x,r)}{\omega_n (\sinh r)^n}d\mathrm{vol}_{\partial \Omega} = \int_{\partial \Omega} \lim_{r\to \infty}\left(\frac{\bar J(x,r)}{\omega_n (\sinh r)^n}\right)\, d\mathrm{vol}_{\partial \Omega}$$
    exists and is finite.
\end{proof}

%%%%%%%%%%%%%%%%%%%%%%%%%%%%%%%%%%%%%%%%%%%%%%%%%%%%%%%%%%%%%%%%%%%%%%%%%%%%%%%%%%%%

\section{Willmore inequality with integral Ricci curvature}
 In this section, we establish a Willmore-type inequality involving integral Ricci curvature (Theorem~\ref{thm:willmore-integral}). 
 Unlike the asymptotic setting treated in Section~2, 
 we now assume only that the function $\rho(x)= \max\{0,-n-\mathrm{Ric}(x)\}$ belongs to $L^p(M)$ for some $p > \frac{n+1}{2}$.
 %here we work under the weaker assumption that $\rho(x)= \max\{0,-n-\mathrm{Ric}(x)\}$ has finite $L^p$ norm for $p>\frac{n+1}{2}$. 
 This integral condition allows for more general geometric settings but requires different techniques.

The key step in our proof is establishing a volume estimate for tubular neighborhoods of hypersurfaces under integral Ricci curvature bounds. Gallot’s result \cite[Theorem 2]{Gallot88} yields a lower bound for a Willmore-type functional with suitable choices of the parameter $\alpha$. However, this lower bound can become negative when the volume of the given hypersurface is large or integral Ricci curvature is large. 
Petersen and Sprouse \cite[Lemma 4.1]{PetersenSprouse} also obtain an upper bound for the volume of tubular neighborhoods of hypersurfaces, but their results requires that the hypersurface has constant mean curvature.
To overcome these limitations, we develop a more refined volume estimate with a suitable Willmore-type functional that remains non-trivial and effective.
 
 Our strategy begins with the derivation of an $L^p$ version of mean curvature comparison in inequality \eqref{ineq:mean-integral-comp}.
 While mean curvature comparison with integral Ricci curvature has been well-studied for geodesic balls in manifolds \cite{petersen-wei}, leading to volume comparisons for geodesic balls \cite{petersen-wei,ChenWei22}, the hypersurface setting presents additional challenges. Specifically, our tubular neighborhood of hypersurface has no constant mean curvature, preventing direct application of existing techniques.
 To address this difficulty, we establish a Jacobian determinant comparison in \eqref{ineq:determinantJacob}.
 We then apply Lemma~\ref{thm:estimate}, which provides a delicate estimate for $(1+b)^p$ that allows us to control error terms arising from integral Ricci curvature. This estimate ultimately yields the desired Willmore-type inequality with an error term $C(n,p,\|\rho\|_p) \to 0$ as $\|\rho\|_p \to 0^+$, recovering the classical result under the pointwise Ricci lower bound $\mathrm{Ric}\geq -n$.

\begin{lem}\label{thm:estimate}
    Let $p>1$, $q>p-1$, and $\epsilon>0$. Then there exists a constant $C(p,q,\epsilon)\geq 0$ such that for all $b \geq 0$, we have
    $$(1+b)^p \leq 1+\epsilon + C(p,q,\epsilon)\epsilon^{-q} b^p.$$
    Moreover, $C(p,q,\epsilon)$ is continuous with respect to $\epsilon$ and $\lim_{\epsilon \to 0^+}C(p,q,\epsilon) = 0$.
\end{lem}

\begin{proof}
    If $b=0$, then the inequality is trivial. Assume $b>0$.
    Consider the function
    $$
        F_{p,q,\epsilon}(b) := \frac{(1 + b)^p - 1 - \epsilon}{\epsilon^{-q} b^p}, \quad b > 0.
    $$
    The function $F_{p,q,\epsilon}$ is continuous on $(0, \infty)$. Since
    $$
        \lim_{b \to 0^+} F_{p,q,\epsilon}(b) = -\infty \quad \text{and} \quad 
        \lim_{b \to \infty} F_{p,q,\epsilon}(b)  = \epsilon^q,
    $$
    the function $F_{p,q,\epsilon}(b)$  is bounded above and attains a finite maximum.
    A critical point of the function $F_{p,q,\epsilon}(b)$ is
    $
        \tilde{b} := (1 + \epsilon)^{\frac{1}{p-1}} - 1 > 0
    $
    and the critical value is
    \begin{align*}
        F_{p,q,\epsilon}(\tilde{b}) 
         = \left(\frac{\epsilon^{\frac{q}{p-1}}}{(1+\epsilon)^{\frac{1}{p-1}} -1} + \epsilon^{\frac{q}{p-1}}\right)^{p-1}.
    \end{align*}
    Note that $\lim\limits_{\epsilon \to 0^+} F_{p,q,\epsilon}(\tilde{b}) = 0$ because
    \begin{align*}
        \lim_{\epsilon\to 0^+}\frac{\epsilon^{\frac{q}{p-1}}}{(1+\epsilon)^{\frac{1}{p-1}} -1} = \lim_{\epsilon\to 0^+} \frac{\frac{q}{p-1}\epsilon^{\frac{q}{p-1}-1}}{\frac{1}{p-1}(1+\epsilon)^{\frac{1}{p-1}-1}} = 0
    \end{align*}
    provided that $q > p - 1 $.
    
    We define the constant
    $$
        C(p,q,\epsilon) := \sup_{b > 0} F_{p,q,\epsilon}(b) = \max\left\{F_{p,q,\epsilon}(\tilde{b}),\, \epsilon^q\right\}.
    $$
    Then $C(p,q,\epsilon)$ is continuous with respect to $\epsilon$ and $\lim\limits_{\epsilon \to 0^+}C(p,q,\epsilon) = 0$.
\end{proof}

With the above estimate in hand, we are now ready to prove the main theorem.

\begin{proof}[Proof of Theorem~\ref{thm:willmore-integral}]
    As the notation and definitions of the relevant geometric objects are the same as in the proof of Theorem~\ref{thm:willmore-asymptotic}, we proceed directly to the differential inequality
        \begin{equation*}
            m'(x,t)  + \frac{1}{n} m(x,t)^2\leq -\mathrm{Ric}(\gamma_x'(t), \gamma_x'(t)),
        \end{equation*}
    with initial condition $m(x,0) = H(x)$, where $H(x)$ denotes the mean curvature of $\Sigma = \partial \Omega$ at $x \in \Sigma$, and $\gamma_x(t)$ is the unit-speed geodesic emanating from $x$ in the outward normal direction.
    Recalling the function $\rho(x) := \max\{-n - \mathrm{Ric}(x),\, 0\}$, we have
    \begin{equation}\label{ineq:riccati}
        m'(x,t) + \frac{1}{n} m(x,t)^2 \leq n + \rho(\gamma_x(t))
    \end{equation}
    for  $0 \leq t < \tau(x)$, where $\tau(x)$ is the focal radius defined in \eqref{def:focalRadius}. 

    We introduce the comparison functions
    $$
        \widehat{m}(x,t) := \frac{n \left(\sinh t + \frac{H(x)}{n} \cosh t\right)}{\cosh t + \frac{H(x)}{n} \sinh t}
        \quad \text{and} \quad
        \widehat{J}(x,t) := \left( \cosh t + \frac{H(x)}{n} \sinh t \right)^n,
    $$
    which correspond to the mean curvature and Jacobian determinant of geodesic spheres in hyperbolic space. The functions satisfy
    $$
        \widehat{m}'(x,t) + \frac{1}{n} \widehat{m}(x,t)^2 = n
        \quad \text{and} \quad
        \frac{\widehat{J}'(x,t)}{\widehat{J}(x,t)} = \widehat{m}(x,t).
    $$
    Consequently, the difference $m(x,t) - \widehat{m}(x,t)$ satisfies
    \begin{align*}
        (m(x,t) - \widehat{m}(x,t))'
        &\leq -\frac{1}{n} \left(m(x,t)^2 - \widehat{m}(x,t)^2\right) + \rho(\gamma_x(t)) \\
        &= -\frac{1}{n} \left(m(x,t) - \widehat{m}(x,t)\right)^2
            - \frac{2}{n} \left(m(x,t) - \widehat{m}(x,t)\right) \widehat{m}(x,t)
         + \rho(\gamma_x(t)).
    \end{align*}
    To measure the deviation from the comparison function $\widehat{m}(x,t)$, we define
    $$
        \phi(x,t) := \max\{m(x,t) - \widehat{m}(x,t),\, 0\}.
    $$
    By construction, $\phi(x,0) = 0$. From the previous differential inequality, we obtain
    $$
        \rho(\gamma_x(t)) \geq \phi'(x,t) + \frac{1}{n} \phi(x,t)^2 + \frac{2}{n} \phi(x,t) \widehat{m}(x,t)
    $$
    for $0 \leq t < \tau (x)$.
    We now multiply both sides by $\phi(x,t)^{2p-2} J(x,t)$, where $J(x,t)$ denotes the extended Jacobian determinant of the normal exponential map defined in \eqref{def:jxt}. This yields
    \begin{align*}
        \rho(\gamma_x(t))\, \phi^{2p-2} J
        &\geq \phi' \phi^{2p-2} J + \frac{1}{n} \phi^{2p} J + \frac{2}{n} \widehat{m} \phi^{2p-1} J \\
        &= \frac{1}{2p - 1} \left( \phi^{2p - 1} J \right)' - \frac{1}{2p - 1} \phi^{2p - 1} m J + \frac{1}{n} \phi^{2p} J + \frac{2}{n} \widehat{m} \phi^{2p - 1} J \\
        &\geq \frac{1}{2p - 1} \left( \phi^{2p - 1} J \right)' + \left( \frac{1}{n} - \frac{1}{2p - 1} \right) \phi^{2p} J + \left( \frac{2}{n} - \frac{1}{2p - 1} \right) \widehat{m} \phi^{2p - 1} J.
    \end{align*}
     The last term in the previous inequality is non-negative under our assumptions $H(x)\geq 0$ and $p > \frac{n+1}{2}$.
     It follows that
    $$
        \rho(\gamma_x(t))\, \phi^{2p-2} J \geq \frac{1}{2p - 1} \left( \phi^{2p - 1} J\right)' + \left( \frac{1}{n} - \frac{1}{2p - 1} \right) \phi^{2p}J .
    $$
    Integrating both sides from $0$ to $r$, we obtain
    \begin{align*}
        \int_0^r \rho(\gamma_x(t))\, \phi^{2p - 2} J \, dt
        &\geq \frac{1}{2p - 1} \phi(x,r)^{2p - 1} J(x,r)
        + \frac{2p - 1 - n}{n(2p - 1)} \int_0^r \phi^{2p} J \, dt \\
        &\geq \frac{2p - 1 - n}{n(2p - 1)} \int_0^r \phi^{2p} J \, dt.
    \end{align*}
    Applying H\"{o}lder's inequality to the left-hand side of the previous inequality, we obtain
        \begin{align*}
        \frac{2p - 1 - n}{n(2p - 1)} \int_0^r \phi^{2p} J \, dt
        &\leq \left( \int_0^r \rho(\gamma_x(t))^{p} J \, dt \right)^{\frac{1}{p}}
        \left( \int_0^r \phi^{2p} J \, dt \right)^{1 - \frac{1}{p}}.
        \end{align*}
    Rearranging terms, we find
    $$
        \frac{2p - 1 - n}{n(2p - 1)} \left( \int_0^r \phi(x,t)^{2p} J(x,t) \, dt \right)^{\frac{1}{p}}
        \leq \left( \int_0^r \rho(\gamma_x(t))^{p} J(x,t) \, dt \right)^{\frac{1}{p}},
    $$
    or equivalently,
    \begin{equation} \label{ineq:mean-integral-comp}
        \int_0^r \phi(x,t)^{2p} J(x,t) \, dt
        \leq \left( \frac{n(2p - 1)}{2p - 1 - n} \right)^p
        \int_0^r \rho(\gamma_x(t))^{p} J(x,t) \, dt.
    \end{equation}
    
    We now consider the derivative of the Jacobian ratio  $J(x,t)/\widehat{J}(x,t)$. A direct computation gives
    $$
        \left( \frac{J(x,t)}{\widehat{J}(x,t)} \right)' 
        = \frac{(m(x,t) - \widehat{m}(x,t)) J(x,t)}{\widehat{J}(x,t)} \leq \phi(x,t) \cdot \frac{J(x,t)}{\widehat{J}(x,t)}.
    $$
    This implies
    $$
        \left( \frac{J(x,t)}{\widehat{J}(x,t)} \right)^{\frac{1}{2p} - 1} \left( \frac{J(x,t)}{\widehat{J}(x,t)} \right)' \leq \phi(x,t) \left( \frac{J(x,t)}{\widehat{J}(x,t)} \right)^{\frac{1}{2p}}.
    $$
    Integrating both sides over $[0, r]$, we obtain
    \begin{align*}
        2p \left[ \left( \frac{J(x,r)}{\widehat{J}(x,r)} \right)^{\frac{1}{2p}} - 1 \right]
        &\leq \int_0^r \phi(x,t) \left( \frac{J(x,t)}{\widehat{J}(x,t)} \right)^{\frac{1}{2p}} \, dt \\
        &\leq \left( \int_0^r \phi(x,t)^{2p} J(x,t) \, dt \right)^{\frac{1}{2p}} \left( \int_0^r \widehat{J}(x,t)^{-\frac{1}{2p - 1}} \, dt \right)^{1 - \frac{1}{2p}},
    \end{align*}
    where the last inequality follows from H\"{o}lder's inequality.

    Under our assumption  $H(x) \geq 0$,
    $$C(n,p) : = \frac{1}{2p}\left( \int_0^\infty \widehat{J}(x,t)^{-\frac{1}{2p - 1}} \, dt \right)^{1 - \frac{1}{2p}} < \infty$$
    Then the previous inequality becomes
    \begin{align*}
         \left( \frac{J(x,r)}{\widehat{J}(x,r)} \right)^{\frac{1}{2p}} - 1 
        \leq C(n,p)
        \left( \int_0^r \phi(x,t)^{2p} J(x,t) \, dt \right)^{\frac{1}{2p}}
    \end{align*}
    and so
    \begin{equation}\label{ineq:determinantJacob}
        J(x,r) \leq \widehat{J}(x,r) \left( 1 + C(n,p)
        \left( \int_0^r \phi(x,t)^{2p} J(x,t) \, dt \right)^{\frac{1}{2p}}\right)^{2p}.
    \end{equation}
   For any $q> 2p - 1$ and $\epsilon>0$, Lemma~\ref{thm:estimate} gives
    \begin{equation*}
        J(x,r) \leq \widehat{J}(x,r) \left( 1 + \epsilon + C(n,p,q,\epsilon) \epsilon^{-q}
        \int_0^r \phi(x,t)^{2p} J(x,t) \, dt \right),
    \end{equation*}
   where $C(n,p,q,\epsilon)$ is continuous with respect to $\epsilon$ and $C(n,p,q,\epsilon)\to 0$  as  $\epsilon \to 0^+$.

   We now estimate the volume of a tubular neighborhood of $\Sigma = \partial \Omega$ in $M \setminus \Omega$.
    Applying the Jacobian estimate derived earlier, we have
    {\allowdisplaybreaks
    \begin{align*}
        \mathrm{vol}&\{x \in M \setminus \Omega : d(x,\Sigma) < R\}\\
        & = \int_\Sigma \int_0^RJ(x,t) \, dt \, d\mathrm{vol}_\Sigma(x)\\
        &\leq \int_\Sigma \int_0^R \widehat{J}(x,r) \left( 1 + \epsilon + C(n,p,q,\epsilon)\, \epsilon^{-q}
        \int_0^r \phi(x,t)^{2p} J(x,t) \, dt \right) dr\, d\mathrm{vol}_\Sigma(x) \\
        &= (1 + \epsilon) \int_\Sigma \int_0^R \widehat{J}(x,r) \, dr\, d\mathrm{vol}_\Sigma(x) \\
        &\quad + C(n,p,q,\epsilon)\, \epsilon^{-q}
        \int_\Sigma \int_0^R \widehat{J}(x,r) \left( \int_0^r \phi(x,t)^{2p} J(x,t)\, dt \right) dr\, d\mathrm{vol}_\Sigma(x)\\
        &\leq (1 + \epsilon) \int_\Sigma \int_0^R \widehat{J}(x,r) \, dr\, d\mathrm{vol}_\Sigma(x) \\
        &\quad + C(n,p,q,\epsilon)\, \epsilon^{-q}
        \int_\Sigma \left( \int_0^R \widehat{J}(x,r)\, dr \right)
        \left( \int_0^R \phi(x,t)^{2p} J(x,t)\, dt \right) d\mathrm{vol}_\Sigma(x)\\
        &\leq (1 + \epsilon) \int_\Sigma \int_0^R \widehat{J}(x,r) \, dr\, d\mathrm{vol}_\Sigma(x) \\
        &\quad + C(n,p,q,\epsilon) \epsilon^{-q}
        \left( \int_0^R \widehat{J}(\xi,r)\, dr \right)
        \left( \int_\Sigma \int_0^R \rho(x)^p J(x,t) \, dt\, d\mathrm{vol}_\Sigma(x) \right),
    \end{align*}}
    where we used the inequality \eqref{ineq:mean-integral-comp} in the last inequality and $\xi \in \Sigma$ be a point at which the mean curvature attains its maximum.

    Using the asymptotic expansion
    $$
        \widehat{J}(x,t) = \left( \cosh t + \frac{H(x)}{n} \sinh t \right)^n
        = \frac{e^{nt}}{2^n} \left(1 + \frac{H(x)}{n} \right)^n + O\left(e^{(n-2)t}\right),
    $$
    we have
    \begin{align*}
        \mathrm{vol}&\{x \in M \setminus \Omega : d(x,\Sigma) < R\}\\
        &\leq (1 + \epsilon) \int_\Sigma \left[ \frac{e^{nR}}{n \cdot 2^n} \left(1 + \frac{H(x)}{n} \right)^n + O\left(e^{(n-2)R}\right) \right] d\mathrm{vol}_\Sigma(x) \\
        &\quad + C(n,p,q,\epsilon) \epsilon^{-q}
        \left[ \frac{e^{nR}}{n \cdot 2^n} \left(1 + \frac{H(\xi)}{n} \right)^n + O\left(e^{(n-2)R}\right) \right]  \int_\Sigma \int_0^R \rho(x)^p J(x,t) \, dt\, d\mathrm{vol}_\Sigma(x).
    \end{align*}
    Dividing both sides of the inequality by $\frac{e^{nR}}{n \cdot 2^n}$ and taking the limit as $R \to \infty$, we obtain
    $$
        \mathrm{RV}(\Omega) \cdot \omega_n
        \leq (1 + \epsilon) \int_\Sigma \left(1 + \frac{H(x)}{n} \right)^n \, d\mathrm{vol}_\Sigma(x)
        + C(n,p,q,\epsilon)\, \epsilon^{-q} \left(1 + \frac{H(\xi)}{n} \right)^n \|\rho\|_p^p.
    $$
    When $\|\rho\|_p \neq 0$, by choosing $\epsilon = \|\rho\|_p^{\frac{1}{2}}>0$ and $q = 2p > 2p - 1$, we obtain
    $$
        \mathrm{RV}(\Omega) \cdot \omega_n
        \leq \left(1 + \|\rho\|_p^{\frac{1}{2}} \right)
        \int_\Sigma \left(1 + \frac{H(x)}{n} \right)^n \, d\mathrm{vol}_\Sigma(x)
        + C(n,p,\|\rho\|_p) \left(1 + \frac{H(\xi)}{n} \right)^n,
    $$
    where $\displaystyle \lim_{\|\rho\|_p \to 0^+} C(n,p,\|\rho\|_p) = 0$.

    When $\|\rho\|_p = 0$, the Ricci curvature satisfies the pointwise lower bound $\mathrm{Ric} \geq -n$, and the above inequality holds by \cite[Theorem 1.3]{jin-yin}, with  $C(n,p,0) :=0$.
    \end{proof}

\bibliographystyle{plain}
\bibliography{willmore}

 \vskip 1cm
 \noindent Jihye Lee\\
 Department of Mathematics\\
University of California, Santa Barbara, CA 93106, USA\\
 {\tt E-mail:jihye@ucsb.edu} \\

\end{document}